\documentclass[12pt]{article}
\usepackage{amsmath,amsfonts,amssymb,amsthm, epsfig,
epstopdf, titling, url, array}
\usepackage[utf8]{inputenc}
\usepackage{mathtools}
\usepackage[english]{babel}
\usepackage{hyperref}
\newtheorem{thm}{Theorem}[section]
\newtheorem{lem}{Lemma}[section]
\newtheorem{cor}{Corollary}[section]
\newtheorem{rmk}{Remark}[section]
\newtheorem{prop}{Proposition}[section]
\numberwithin{equation}{section}

\def\R{I\!\!R}
\def\H{I\!\!H}

\title{POISSON AND HEAT SEMIGROUPS FOR THE  BESSEL OPERATOR AND ON THE HYPERBOLIC SPACE}
\author{Adam Zakria, Ibrahim-Elkhalil Ahmed and Mohamed Vall Ould Moustapha}

\begin{document}
\maketitle
\begin{abstract}
In this paper we find explicit formulas for the Poisson and heat semigroups associated to the modified Bessel operator and on the hyperbolic spaces $\H^n$.
 
\end{abstract}
 
\section{Introduction}
The differential operators of Bessel type and the Laplace-Beltrami operator on the hyperbolic space are known as very important operators
in analysis and its applications.
This paper deals with the Poisson and heat semigroups associated to these second
order differential operators. In
the last decades the Poisson and heat semigroups
 associated to  many second differential operators 
have been studied and computed explicitly and there is many interesting papers published in this area of reaserch(see for example Betancor et al.\cite{BETANCOR et al.}, Isolda Cardoso\cite{ISLODA-CARDOSO}, Keles and Bayrakci\cite{SEIDA-SIMTEN}, Stein \cite{STEIN} and the references theirin).
The main objective of this paper is to solve explicitly the following Poisson and heat problems
\begin{align}\label{Poisson-Bessel-Problem} \left\{\begin{array}{cc}L^{a} u(y, x)=-\frac{\partial^2}{\partial
y^2}u(y, x)
, (y, x)\in \R^+\times \R \\ u(0, x)=u_0(x), u_0\in
C^\infty_0(\R^+)\end{array}
\right. ,\end{align}
\begin{align}\label{Poisson-Hyperbolic-Problem} \left \{\begin{array}{cc}{\cal L}_n U(y, w)=-\frac{\partial^2}{\partial
y^2}U(y, w)
, (y, w)\in \R^+\times \H^n \\ U(0, w)=U_0(w), U_0\in
C^\infty_0(\H^n)\end{array}
\right., \end{align}
and
\begin{align}\label{Heat-Bessel-Problem} \left \{\begin{array}{cc}{L}^{a} v(t, x)=\frac{\partial}{\partial
t}v(t, x)
, (t, x)\in \R^+\times \R \\ v(0, x)=v_0(x), v_0\in
C^\infty_0(\R^+)\end{array}
\right. ,\end{align}
\begin{align}\label{Heat-Hyperbolic-Problem} \left \{\begin{array}{cc}{\cal L}_n V(t, w)=\frac{\partial}{\partial t}V(t, w)
, (t, w)\in \R^+\times \H^n \\ V(0, w)=V_0(w), V_0\in
C^\infty_0(\H^n)\end{array}
\right. ,\end{align}
where 
\begin{align}\label{Bessel-Operator}L^{a}=x^2\frac{\partial^2}{\partial x^2}+x\frac{\partial}{\partial x}-a^2x^2, \end{align}
and
\begin{align}\label{Beltrami-Operator}{\cal L}_n=x^2_n\Delta_{n-1}+x_n^2\frac{\partial^2}{\partial x_n^2}+(2-n)x_n\frac{\partial}{\partial x_n}+\frac{(n-1)^2}{4}, \end{align}
are respectively the Bessel operator on $\R^+$ and the Laplace-Beltrami operator on the half space model of the hyperbolic space $\H^n$.
\section{Poisson semigroup associated to Bessel operator}
In this section we give explicit formulas for the Poisson semigroup associated to the Bessel operator $L^a$, that is we prove the following theorem.
\begin{thm}\label{Poisson-Bessel} For $a\in \R^*$ the Poisson problem \eqref{Poisson-Bessel-Problem} has the solution given by
\begin{align}\label{u}u(y, x)=\int_0^\infty p_a(y, x, x')u_0(x') \frac{d x'}{x'},\end{align}
with
\begin{align} \label{pa}p_a(y, x, x')=\frac{|a|}{\pi}\frac{x x'\sin y K_1 \left(|a|\sqrt{x^2+x'^2-2x x'\cos y}\right)}{\sqrt{x^2+x'^2-2x x'\cos y}}, \end{align}
and $K_1$ is the modified Bessel functions of second kind.
\end{thm}
\begin{proof}
To see that the function $u(y, x)$ satisfes the Poisson equation in \eqref{Poisson-Bessel-Problem}, set
$\varphi(y, x, x')=\phi(z)$, with
$z=x^2+x'^2-2x x'\cos y$, then we have
$$\frac{\partial \varphi}{\partial x}=(2x-2x'\cos y)\frac{\partial \phi}{\partial z},\ \ \frac{\partial^2 \varphi}{\partial x^2}=(2x- 2x'\cos y)^2\frac{\partial^2 \phi}{\partial z^2}+ 2\frac{\partial \phi }{\partial z},
 $$
and 
$$\frac{\partial \varphi}{\partial y}=2x x'\sin y\frac{\partial \phi}{\partial z},\ \ 
\frac{\partial^2 \varphi}{\partial y^2}=(2x x'\sin y)^2\frac{\partial^2\phi}{\partial z^2}+2x x'\cos y\frac{\partial \phi}{\partial z}. $$
Using the above formulas we have
\begin{align*}
\left(L^a+\frac{\partial^2}{\partial y^2}\right)\varphi=4x^2\left(z\frac{\partial^2 \phi}{\partial z^2}+\frac{\partial \phi}{\partial z}-\frac{a^2}{ 4}\phi\right),\end{align*}
and we see that the first equation in the problem \eqref{Poisson-Bessel-Problem} is equivalent to 
$$ z^2\phi_{zz}+z\phi_z-\frac{a^2}{4}z\phi=0,$$
which is a particular case of Lommel differential equation for modified Bessel functions
$$[z^2\frac{\partial^{2}\phi}{\partial z^{2}}+(1-2\alpha)z\frac{\partial \phi}
{\partial z}-(\beta\gamma z^{\gamma})^{2}\phi+(\alpha^{2}-\nu^{2}\gamma^{2})\phi=0,$$
with $\alpha=0$, $\nu=0$, $\beta=1$ and $\gamma=1/2$, an approprite solution is $\phi(z)=c K_0(z^{1/2}),$
where $K_0$ is the modified Bessel function of second kind.\\
 This means that the function  $\varphi(y, x, x')= c K_0\left(|a|\sqrt{x^2+x'^2-2x x'\cos y}\right)$ satisfies the equations
$$L^a_x\varphi(y, x, x')=L^a_{x'}\varphi(y, x, x')=-\frac{\partial^2}{\partial y^2}\varphi(y, x, x'),$$
and in consequence it is a solution of the first equation in \eqref{Poisson-Bessel-Problem}.\\
Using the formula $K_0'(z)=-K_1(z)$ we see that
\begin{align}p_a(y, x, x')=-\frac{1}{\pi}\frac{\partial}{\partial y}\varphi(y, x, x')\end{align}
and $p_a(y, x, x')$ satisfies the same equation in \eqref{Poisson-Bessel-Problem}.\\
To finish the proof of Theorem \ref{Poisson-Bessel} it remains to show the limit condition. For this set $z=x^2+x'^2-2x x'\cos y=2x x'(\frac{x^2+x'^2}{2 x x'}-\cos y)$ and $x=e^X$ and $x'=e^{X'}$ to obtain $z=4e^{X+X'}\{\sinh^2\frac{(X-X')}{2}+\sin^2 (y/2)\}$.\\
 Replacing in \eqref{u} we obtain
$$u(y, x)= \tilde{u}(y, X)=\int_0^\infty P_a(y, X, X')\tilde{u}_0(X') dX',$$
with
$$ P_a(y, X, X')=\frac{|a|}{\pi}\frac{e^{(X+X')/2}\sin y K_1\left(|a|\sqrt{\sinh^2\frac{(X-X')}{2}+\sin^2 y/2}\right)}{\sqrt{\sinh^2\frac{(X-X')}{2}+\sin^2 y/2}}.$$
Setting $\sinh\frac{(X'-X)}{2}=s \sin y/2$ or
$X'= X+2arg\sinh(s \sin y/2)$ we can write\\
$ u(y, x)=\tilde{u}(y, X)= \frac{|a|}{\pi}\int_{-\infty}^{\infty} e^{X+ arg\sinh(s \sin y/2)}\sin y\times$ 
$$\frac{K_1\left(2|a|e^{X+arg\sinh(s \sin y/2)}\sin (y/2) \sqrt{1+s^2}\right)}{\sqrt{1+s^2}}\frac{\tilde{u}_0(X+arg\sinh(s \sin y/2))  2 ds}{\sqrt{1+s^2\sin^2 (y/2)}}.$$
Now we use the asymptotic formula for the modified Bessel function of second kind (Lebedev
\cite{LEBEDEV} p.136) $K_\nu(z)\sim \frac{2^{\nu-1}\Gamma(\nu)}{z^\nu}, z\longrightarrow 0$
we obtain
$$\lim_{y\longrightarrow 0}u(y, x)=\lim_{y\longrightarrow 0}\tilde{u}(y, X)=\tilde{u}_0(X)\frac{1}{\pi}\int_{-\infty}^\infty\frac{ds}{1+s^2}=\tilde{u}_0(X)=u_0(x)$$
and this finishes the proof of Theorem \ref{Poisson-Bessel}.
\end{proof}
\section{Poisson equation on the hyperbolic space}
In this section we consider the Poisson equation on the hyperbolic upper half space.\\
 Let $\H^n=\{w=(x_1, x_2,...x_n)\in \R^n, x_n>0\}$ be the hyperbolic half space endowed with the usual hyperbolic metric
\begin{align}ds^2=\frac{dx_1^2+dx_2^2+...+dx_n^2}{x_n^2},\end{align}
the metric $ds$ is invariant with respect to the motion  group $G=SO(n, 1),$
the hyperbolic volume form $d\mu(w)$ is
\begin{align}d\mu(w)=\frac{dx_1dx_2...dx_n}{x^n_n}, \end{align}
 and the hyperbolic distance $\rho(w, w')$ given as
\begin{align}\label{distance} \cosh^2 (\rho(w, w')/2)=\frac{|w-w'|^2}{4x_n x_n'}+1, \end{align} 
with the Laplace Beltrami operator
\begin{align}{\cal L}_n=x_n^2\Delta_n+(2-n)\frac{\partial}{\partial x_n}+((n-1)/2)^2,\end{align}
where $\Delta_n=\sum_{j=1}^n\frac{\partial^2}{\partial x_j^2}$ is the Euclidean Laplacian on $\R^n$.
Before giving the main result of this section we start by  the following lemma in which we compute  the Fourier transform of 
the Poisson semigroup $p_{|\xi|}$ for the Bessel operator with respect to the parameter $|\xi|.$
\begin{lem}\label{Fourier-Poisson} Set $\xi=(\xi_1, \xi_2,...,\xi_{n-1})$ and $x=(x_1, x_2,..., x_{n-1})$, let $p_{|\xi|}(y, x_n, x_n')$ be the  kernel of the Poisson semigroup for Bessel operator given in \eqref{pa}, then the following formula holds.
\begin{align}
{\cal F}^{-1} \left[p_{|\xi|}(y, x_n, x_n')\right](x)=\frac{2^{(n-1)/2}\Gamma((n+1)/2)}{\pi}\frac{ x_n x'_n \sin y}{(z^2+|x|^2)^{(n+1)/2}},
 \end{align}
with $z=\sqrt{x^2+x'^2-2x x'\cos y}.$
\end{lem}
\begin{proof}
From the formula giving the Fourier transform of a radial function
\begin{align}{\cal F}^{-1}[f](|x|)=|x|^{1-n/2} \int^{+\infty}_{0}J_{\frac{n-2}{2}}(\rho|x|)f(\rho)\rho^{n/2}d\rho,\end{align}
we obtain \\
$
{\cal F}^{-1}\left[P_{|\xi|}(y, x_n, x_n')\right](x)=\frac{x_n x'_n \sin y}{\pi z}|x|^{3-n)/2} 
\int^{+\infty}_{0}K_1(\rho|z)J_{\frac{n-3}{2}}(\rho |x|)\rho^{\frac{n+1}{2}}d\rho.$\\
Using formula Prudnikov(\cite{PRUDNIKOV et al.} $p.365$)
\begin{align}\label{Hankel1}\int^{+\infty}_{0}x^{\alpha-1}J_{\mu}(b x)K_{\nu}(c x)d x= A_{\mu,\nu}^\alpha,\end{align}
where
\begin{align}A_{\mu,\nu}^\alpha=2^{\alpha-2}b^\mu c^{-(\alpha+\mu)}
\frac{\Gamma((\alpha+\mu+\nu)/2 \Gamma((\alpha+\mu-\nu)/2)}{\Gamma(\mu+1)}\times\nonumber \\
 _{2}F_{1}((\alpha+\mu+\nu)/2,(\alpha+\mu-\nu)/2,\mu+1,-\frac{b^{2}}{c^{2}}),\end{align}
with $\alpha=(n+3)/2,          \mu=(n-3)/2,      \nu=1, b=|x|,       c=z$ we have
\begin{align} \int^{+\infty}_{0}K_1(r|z)J_{\frac{n-3}{2}}(r\left|x\right|)r^{\frac{n+1}{2}}dr=A^{(n+3)/2}_{1,(n-3)/2},\end{align}
\begin{align}A^{(n+3)/2}_{1,\, (n-3)/2}=\frac{|x|^{(3-n)/2}}{\pi z} 2^{(n-1)/2}\frac{|x|^{(n-3)/2}}{z^n}\Gamma((n+1)/2)\times\\
F\left(((n+1)/2,\, (n-1)/2 \,(n-1)/2\,  -\frac{|x|^2}{z^2}\right).
\end{align}
Now from the formula
$F\left(a, b, b, z\right)=\left(1-z\right)^{-a}$
we obtain the result of Lemma \ref{Fourier-Poisson}.
\end{proof}
\begin{thm} \label{Poisson-Hyperbolic} The Poisson problem \eqref{Poisson-Hyperbolic-Problem} in hyperbolic space $\H^n$ has the solution
given by
\begin{align}\label{U}U(y, w)=\int_0^\infty P_n^{\H}(y, w, w')U_0(w') d \mu(w'),
\end{align}
with 
\begin{align}\label{Poisson-Hyperbolic-Formula} P_n^{\H}(y, w, w')=\frac{\Gamma((n+1)/2)}{\pi^{(n+1)/2}}\frac{\sin y}{\left(2 \cosh d(w,  w')- 2 \cos y\right)^{(n+1)/2}}.
  \end{align}
 \end{thm}
 \begin{proof}
  Using the following formula intertwining the Laplace Beltrami operator ${\cal L}_n$ on the hyperbolic space $\H^n$
  and the Bessel operator $L^{|\xi|}$
\begin{align}{\cal F}\left[x_n^{-(n-1)/2}{\cal L}_n x_n^{(n-1)/2}\phi\right](\xi)=L^{|\xi|}{\cal F}\phi(\xi).
  \end{align}
The Poisson problem on the hyperbolic space \eqref{Poisson-Hyperbolic-Problem} is transformed into the Bessel Poisson problem  \eqref{Poisson-Bessel-Problem}, with $u(y, x_n)={\cal F}\left[x^{(1-n)/2}U(y, x, x_n)\right](\xi)$
and $u_0(x_n)=x_n^{(1-n)/2}{\cal F}[U_0( x, x_n)](\xi)$\\

${\cal F}\left[x^{(1-n)/2}U(y, x, x_n)\right](\xi)=$ $$\int_0^{\infty}P_{|\xi|}(y, x_n, x_n') x_n'^{(1-n)/2}{\cal F} \left[U_0\right](\xi, x_n') \frac{d x_n'}{x_n'},$$
 
\begin{align}
U(y, x, x_n)=\int_0^{\infty}{\cal F}^{-1}\left[P_{|\xi|}(y, x_n, x_n') x_n'^{(1-n)/2}{\cal F} \left[U_0\right](\xi, x_n')\right](x) \frac{d x_n'}{x_n'},
 \end{align}  
 
$U(y, x, x_n)=(2\pi)^{-(n-1)/2}\times$ \begin{align}
 \int_0^{\infty} {\cal F}^{-1}\left[P_{|\xi|}(y, x_n, x_n')\right](x) *U_0(x, x_n') x_n'^{(1-n)/2} \frac{d x_n'}{x_n'},\end{align} 
  $
U(y, x, x_n)=(2\pi)^{-(n-1)/2}\int_0^\infty \int_{\R^{n-1}} {\cal F}^{-1}\left[P_{|\xi|}(y, x_n, x_n')\right](x-x')\times$ $$ U_0(x', x_n')
  x_n^{(n-1)/2} x_n'^{(n-1)/2} \frac{d x' d x_n'}{x_n'^n},$$
  
\begin{align}
 U(y, x, x_n)=\int_{\H^n} P_n^{H}\left(y, w, w'\right)u_0(w')d\mu(w'),
 \end{align}
 with
\begin{align}   
 P_n^{\H}\left(y, w, w'\right)=\frac{\Gamma((n+1)/2)}{\pi^{(n+1)/2}}\frac{\sin y}{\left(2\cosh d(w, w')-2\cos y\right)^{(n+1)/2}}.
 \end{align}
and the proof of Theorem \ref{Poisson-Hyperbolic} is finished.
\end{proof}
\begin{prop}\label{Recurrence-Poisson}
Let $P^{\H}_n\left(y, w, w'\right)$ be the Poisson kernel on the hyperbolic space $\H^n$ then we have\\
i)$ \left(-\frac{\partial}{2\pi\sinh \rho\partial \rho}\right)P^{\H}_n\left(y, \rho(w, w')\right)=P^{\H}_{n+2}\left(y, \rho(w, w')\right),$\\
ii) $ \int_r^\infty \frac{P^{\H}_{n+1}\left(y, \rho\right)}{\sqrt{\cosh^2\rho/2-\cosh^2 r/2}}\sinh\rho d\rho =P^{\H}_{n}\left(y, r\right)$.\\
\end{prop}
\begin{proof}
The part i) is simple, to prove ii) set
$$ I=\int_r^{\infty} C_{n+1}\frac{\sin y}{(\cosh\rho - cos y)^{(n+2)/2}}(\cosh^2\rho/2 - \cosh^2 r/2)^{-1/2}\sinh r d\rho, $$
$$ I=\sqrt{2}\int_r^{\infty} C_{n+1}\frac{\sin y}{(\cosh\rho - cos y)^{(n+2)/2}}(\cosh\rho - \cosh r)^{-1/2}\sinh r d\rho, $$
with $C_{n+1}=\frac{\Gamma(n+2)/2)}{(2\pi)^{(n+2)/2}}$. Set $\cosh\rho=\sigma$ and set $\sigma-\cosh r=\mu$, we see that
$$I=\sqrt{2}C_{n+1}\int_{\cosh r}^{\infty}\sin y\left(\sigma-\cos y\right)^{-(n+2)/2}\left(\sigma-\cosh r\right)^{-1/2}d\sigma,$$

$$I=\sqrt{2}C_{n+1}\sin y\int_{0}^{\infty}\left(\mu-(\cos y-\cosh r)\right)^{-(n+2)/2}\mu^{-1/2}d\mu,$$
$I=\sqrt{2}C_{n+1}\sin y \left(\cosh r-\cos y\right)^{-(n+2)/2}\times $
$$\int_{0}^{\infty}\left(1+\frac{1}{(\cosh r-\cos y)}\mu \right)^{-(n+2)/2}\mu^{-1/2}d\mu.$$
Using the formula (Magnus et al.\cite{MAGNUS et al.} p.13)
$$\int_0^\infty t^{x-1}(1+b t)^{-x-y}dt= b^{-x}B(x, y),$$
where $B$ is the beta function,
with $x=1/2$ and $y=(n+1)/2$ we obtain
$$I=\sqrt{2}C_{n+1}\sin y \left(\cosh r-\cos y\right)^{-(n+1)/2}B(1/2, (n+1)/2),$$
$$I=\sqrt{2}\frac{\Gamma((n+1)/2)}{2^{(n+2)/2}\pi^{(n+1)/2}}\sin y \left(\cosh r-\cos y\right)^{-(n+1)/2}=P^{\H}_{n}\left(y, r\right),$$  
thus we obtain ii) and the proof of Theorem \ref{Recurrence-Poisson}is finished.
\end{proof}
\section{Heat semigroup on hyperbolic space}
In this section we give a new explicit formula for the heat kernel on the hyperbolic space $\H^n$.
\begin{prop}\label{Recurrence-Heat}
Let $e^{-y\sqrt{-{\cal L}_n}}$ and $e^{t {\cal L}_n}$ be the Poisson and heat semigroups on the hyperbolic space $\H^n$ then we have\\
 i) $e^{t {\cal L}_n}=(4t)^{-1/2}L_{y^2}^{-1}\left[\frac{\sqrt{\pi}e^{-y\sqrt{-{\cal L}_n}}}{y}\right](1/4 t)$, where $L_{y^2}^{-1}$ is the Laplace inverse transform with respect to $y^2$.\\
ii) $ \left(-\frac{\partial}{2\pi\sinh \rho\partial \rho}\right)K^{\H}_n\left(t, \rho(w, w')\right)=K^{\H}_{n+2}\left(t, \rho(w, w')\right),$\\
iii) $ \int_r^\infty \frac{K^{\H}_{n+1}\left(t, \rho\right)}{\sqrt{\cosh^2\rho/2-\cosh^2 r/2}}\sinh\rho d\rho =K^{\H}_{n}\left(t, r\right).$\\
\end{prop}
\begin{proof}
To prove i) use the subordination formula (Strichartz \cite{STRICHARTZ} $p. 50$).\\
$\frac{e^{-y \lambda}}{y}=\frac{1}{\sqrt{\pi}}\int_0^\infty e^{-u y^2}u^{-1/2}e^{-\lambda^2/4u}du$ or
$\frac{\sqrt{\pi}e^{-y \lambda}}{y}=L\left(u^{-1/2}e^{-\lambda^2/4u}\right)(y^2),$\\
where $(Lf)(p)$ is the Laplace transform
and
$e^{\frac{-\lambda^2}{4 u}}=u^{1/2}L^{-1}_{y^2}\left(\sqrt{\pi}\frac{e^{-y \lambda}}{y}\right)(u)$. \\
Set $\lambda=\sqrt{{\cal L}_n}$ and $\frac{1}{4u}=t $ in the last formula we can write
$$ e^{t {\cal L}_n}=(4t)^{-1/2}L_{y^2}^{-1}\left[\frac{\sqrt{\pi}e^{-y\sqrt{-{\cal L}_n}}}{y}\right](1/4 t),
$$
where $L^{-1}$ is the inverse Laplace transform. \\
The parts ii) and iii) are consequence of i) and Proposition \ref{Recurrence-Poisson}
\end{proof}

\begin{thm}\label{Heat-Hyperbolic} The heat Cauchy problem on hyperbolic space \eqref{Heat-Hyperbolic-Problem} has the unique solution 
given by
\begin{align}\label{Solution-Heat-Hyperbolic}V(t, w)=\int_H K_n(t, w, w')V_0(w') du(w'),\end{align}
with
\begin{align}\label{Heat-Kernel} K_n(t, w, w')=\frac{\Gamma((n+1)/2)}{2^{(n+1)/2}\pi^{n/2}t^{1/2}}\int_{\sigma-i\infty}^{\sigma+i\infty}\frac{\exp{\left(\frac{y^2}{4t}\right)}\, \sin y}{(\cosh\rho(w, w')-\cos y)^{(n+1)/2}}dy.\end{align}
\end{thm}
\begin{proof}
using \eqref{Poisson-Hyperbolic-Formula} we see the formula \eqref{Heat-Kernel} 
and the proof of Theorem\ref{Heat-Hyperbolic} is finished.

\end{proof}

\begin{cor} (Davies-Mandouvalos\cite{DAVIES-MANDOUVALOS} and Lohoue and Rychener\cite{LOHOUE-RYCHENER}) Let $K_n(t, w, w')$ be the heat kernel on the hyperbolic space $\H^n$ then we have\\
i) For  $n$ odd\ \ \
 $K_n(t, w, w')=\left(-\frac{\partial}{2\pi\sinh \rho\partial \rho}\right)^{\frac{n-1}{2}} \frac{e^\frac{-\rho^2}{4t}}{(4\pi t)^{1/2}}$,\\
ii) for $n$ even \ \ 
$K_n(t, w, w')=\left(-\frac{\partial}{2\pi\sinh \rho\partial \rho}\right)^{\frac{n-2}{2}}\int_\rho^\infty\left(\cosh^2 s/2 -\cosh^2 \rho/2 \right)^{-1/2} \frac{e^\frac{-s^2}{4t}}{(4\pi t)^{3/2}}sds. $\\

 \end{cor}
 \begin{proof}
Set $\cos y=z$
$$K_n(t, w, w')=\frac{\Gamma((n+1)/2)}{2^{(n+1)/2}\pi^{n/2}t^{1/2}}\int_{\sigma-i\infty}^{\sigma+i\infty} \frac{e^{\frac{(\arccos z)^2}{4t}}}{(\cosh\rho(w, w')-z)^{(n+1)/2}}dz.$$
To prove the first statment i) we have
$$K_1(t, w, w')=\frac{1}{2 \pi^{1/2}t^{1/2}}\int_{\sigma-i\infty}^{\sigma+i\infty}\frac{e^{\frac{(\arccos z)^2}{4t}}}{(\cosh\rho(w, w')-z)}dz,$$
that is
$$K_1(t, w, w')=\frac{1}{2 \pi^{1/2}t^{1/2}}Res_{z=\cosh\rho}\left[\frac{e^{\frac{(\arccos z)^2}{4t}}}{(\cosh\rho(w, w')-z)}\right],$$
and
\begin{align}\label{Heat1}K_1(t, w, w')=\frac{1}{\sqrt{4\pi t}}e^{-\frac{\rho^2}{4 t}},\end{align}
using ii) of \ref{Recurrence-Heat} we have i).\\
To prove iii) we can write
$$K_3(t, w, w')=\frac{1}{2^2\pi^{3/2}t^{1/2}}\int_{\sigma-i\infty}^{\sigma+i\infty} \frac{e^{\frac{(\arccos z)^2}{4t}}}{(\cosh\rho(w, w')-z)^2}dz,$$
and
$$K_3(t, w, w')=\frac{1}{2^2\pi^{3/2}t^{1/2}}Res_{z=\cosh\rho}\left[\frac{e^{\frac{(\arccos z)^2}{4t}}}{(\cosh\rho(w, w')-z)^2}\right],$$
this gives
$$K_3(t, w, w')=\frac{1}{2^2\pi^{3/2}t^{1/2}}\lim_{z\longrightarrow \cosh\rho}\frac{d}{d z}\left[e^{\frac{(\arccos z)^2}{4t}}\right],$$
and finally 
\begin{align}\label{Heat3} K_3(t, w, w')=\frac{1}{(4\pi t)^{3/2}}\frac{\rho}{\sinh\rho}e^{\frac{-\rho^2}{4t}},\end{align}
using iii) of Theorem \ref{Recurrence-Heat} we have
 \begin{align}\label{Heat2} K_2(t, w, w')= \int_\rho^\infty\left(\cosh^2 s/2 -\cosh^2 \rho/2 \right)^{-1/2} \frac{e^\frac{-s^2}{4t}}{(4\pi t)^{3/2}}sds,\end{align}
  
 Combining \eqref{Heat2} and the part i) of Proposition \ref{Recurrence-Heat} we obtain ii) and the proof of Corollary \ref{Heat-Kernel} is finished.
\end{proof}
Note that the wave equation on hyperbolic space is studied in {intissar-Ould Moustapha}\cite{INTISSAR-OULD MOUSTAPHA} Bunk et al.\cite{BUNK et al.} Lax-Phillips \cite{LAX-PHILLIPS}
\section{Heat kernel for the Bessel operator}
\begin{prop}\label{prop-Heat}
i) The modified Laplace-Beltrami operator ${\cal L}_2$ on the hyperbolic space  and  Bessel operator $L^a$ on  $\R^+$  are connected via the formulas
\begin{align}{\cal F}_{x_1}\left[x_2^{-1/2}{\cal L}_2 x_2^{1/2}\Phi\right](a, x_2)=L^a\left({\cal F}\Phi\right)(\lambda, x_2),\end{align}
where the Fourier transform is given by
\begin{align}
[{\cal F}f](\xi)=\frac{1}{\sqrt{2\pi}}\int_{\R}e^{-i x \xi}f(x)dx.
\end{align}
ii) The heat kernels for Bessel operator  $H^{a}(t, x_2, x_2')$  is connected to the heat kernel on the hyperbolic half plane $H_2(t, z, z')$ via the formula
\begin{align}\label{Intertwining-Heat} H_{a}(t, x_2, x_2')= \frac{1}{\sqrt{x_2 x_2'}}\int_{-\infty}^{\infty}e^{-i a(x_1-x_1')}H_2(t, z, z')d(x_1-x_1').\end{align}
\end{prop}
\begin{proof}
The proof of this proposition is simple and in consequence is left to the reader.
\end{proof}
\begin{thm}\label{Heat-Bessel} The heat Cauchy  problem for the Bessel operator $ L^a $ has the unique solution given by 
\begin{align} u(t, x)=\int_{\R} K_a(t, x, x')u_0(x') dx', 
\end{align}
with \\
$K_a(t, x, x')=\frac{1}{4\sqrt{\pi} t^{3/2}}\times$
$$\int_{\cosh^2s/2\geqslant\frac{x^2+x'^2}{4xx'}}^\infty se^\frac{-s^2}{4t}J_0(|a|\sqrt{4xx'\cosh^2s/2-x^2-x'^2})ds.$$
 \end{thm}
\begin{proof}
 The  proof of this theorem follows from the  proposition \ref{Intertwining-Heat}, the Fubini theorem and the  formula  (Lebedev\cite{LEBEDEV} p.114)
\begin{align} J_0(z)=\frac{1}{[\Gamma(1/2)]^2}\int_{-1}^{1} (1-t^2)^{-1/2}\cos z tdt=\frac{1}{\pi}\int_{-1}^{1} (1-t^2)^{-1/2}e^{-izt} tdt. \end{align}
\end{proof}

\section{Applications}
In this section we give some applications of our results.
As an application of Theorem \ref{Poisson-Bessel} and \ref{Heat-Bessel} we give the following corollary giving explicit solution to the Poisson and heat problems with Morse potential.
For recent work on Morse potential the reader can consult ( Abdelhaye et al. \cite{A-B-M}, Ikeda-Matsumoto\cite{IKEDA-MATSUMOTO}, Morse \cite{MORSE} and Ould Moustapha \cite{OULD MOUSTAPHA}).


\begin{cor}\label{Poisson-Morse} For $a\in \R$ the problem 
\begin{align}\label{Poisson-Morse-Problem} \left\{\begin{array}{cc}{M}^{a}  \tilde{u}(y, X)=-\frac{\partial^2}{\partial
y^2}\tilde{u}(y, X),(y, X)\in \R^+\times \R \\ \tilde{u}(0, X)=\tilde{u}(X)_0, U_0\in
C^\infty_0(\R^+)\end{array}
\right. ,\end{align}
 has the solution given by
\begin{align}\tilde{u}(y, X)=\int_0^\infty P_a(y, X, X')\tilde{u}_0(X') d X',\end{align}
with
\begin{align} \tilde{P}_a(y, X, X')=\frac{|a|}{\pi}\frac{e^{X+X'}\sin y K_1\left(|a|\sqrt{\sinh^2\frac{(X-X')}{2}+\sin^2 y/2}\right)}{\sqrt{\sinh^2\frac{(X-X')}{2}+\sin^2 y/2}},\end{align}
where $K_1$ is the modified Bessel functions of second kind.
\end{cor}
\begin{proof} Set $X=\ln x$ the problem \eqref{Poisson-Morse-Problem} is transformed into the problem \eqref{Poisson-Bessel-Problem} and it is not hard to see the result of theorem from \ref{Poisson-Bessel}.
\end{proof}
\begin{thm} The heat Cauchy  problem with Morse Potential
\begin{align}\label{Heat-Morse-Problem} \left\{\begin{array}{cc}{M}^{a}\tilde{V}(t, X)=\frac{\partial}{\partial t}\tilde{V}(t, X)
, (t, X)\in \R^+\times \R \\ \tilde{V}(t, X)=\tilde{V}(X)_0,\tilde{V}_0\in
C^\infty_0(\R^+)\end{array}
\right. , \end{align}
 has the unique solution given by 
\begin{align} \tilde{V}(t, X)=\int_R K_a(t, X, X')\tilde{V}_0(X') dX', 
\end{align}
with \\
$ K_a(t, X, X')=\frac{1}{4\sqrt{\pi} t^{3/2}}\times$
$$\int_{|X-X'|}^\infty se^\frac{-s^2}{4t}J_0(2|a|e^{(X+X')/2}\sqrt{\cosh^2s/2-\cosh^2((X-X')/2)})ds
  $$
 \end{thm}
\begin{proof} Set $X=\ln x$ the problem \eqref{Heat-Morse-Problem} is transformed into the problem \eqref{Heat-Bessel-Problem} and it is not hard to see the result of theorem from Theorem \ref{Heat-Bessel}.
\end{proof}
\begin{rmk} The Poisson and heat semigroups of the operator of bessel type 
\begin{align}\label{Bessel Operator L1 }L_{\alpha}=x^2\frac{d^2}{d x^2}+(2\alpha+3)\frac{d}{d x}+x^2+(\alpha+1)^2 \end{align}
are considered in Betancor et al. \cite{BETANCOR et al.}.\\
It is not hard to see that
\begin{align}\label{Bessel Operator L2 } x^{\alpha+1}L_{\alpha}x^{-\alpha-1}=L^{-i} =x^2\frac{\partial^2}{\partial x^2}+x\frac{\partial}{\partial x}+x^2.\end{align}
\end{rmk}

\begin{cor}\label{cor1} If $a\in i\R^*$ and $a=ib$ the problem \eqref{Poisson-Bessel} has the solution given by
\begin{align}\label{v}v(y, x)=\int_0^\infty q_b(y, x, x')v_0(x') \frac{d x'}{x'},\end{align}
with 
\begin{align} q_b(y, x, x')=-1/2|b|\frac{x x'\sin y H^{(1)}_1\left(|b|\sqrt{x^2+x'^2-2x x'\cos y}\right)}{\sqrt{x^2+x'^2-2x x'\cos y}}, \end{align}
\begin{align}q_b(y, x, x') =-1/2 \frac{\partial}{\partial y}H^{(1)}_0\left(|b|\sqrt{x^2+x'^2-2x x'\cos y}\right),\end{align}
wih $H^{(1)}_1$, $H^{(1)}_0$ are the Bessel function of the third kind.
\end{cor}
\begin{proof}
Using formula $K_\nu(ze^{\frac{-i\pi\nu}{2}})=1/2 i\pi e^{\frac{i\pi\nu}{2}}H^{(n)}_\nu(z)$ (Magnus et al.\cite{MAGNUS et al.}$ p67.$)
\end{proof}
The last application of our result is the explicit formula of the Poisson semigroup on the sphere $S^{n}$. 
\begin{cor}\label{Poisson-Spheric} The Poisson equation in the sphere $S^n$ has the unique solution given by  
\begin{align}
 u(y, \omega)=\int_{S^n} P_n^{S}\left(y, \omega, \omega'\right)u_0(\omega')d\mu(\omega'),
 \end{align}
 with
\begin{align}  
 P_n^{S}\left(y, \omega, \omega'\right)=\frac{\Gamma((n+1)/2)}{\pi^{(n+1)/2}}\frac{\sinh y}
 {\left(2\cosh y-2\cos d(\omega, \omega')\right)^{(n+1)/2}},
\end{align}
\end{cor}
\begin{proof}
By comparing the radial parts of the Laplace Beltrami operators on the spaces $\H^n$ and $S^n$ given respectively by
$$\Delta^{\H^n}=\frac{\partial^{2}}{\partial r^{2}}+(n-1)\coth r\frac{\partial}{\partial r} +(\frac{n-1}{2})^{2}$$
and
$$\Delta^{S^n}=\frac{\partial^{2}}{\partial r^{2}}+(n-1)\cot r\frac{\partial}{\partial r} -(\frac{n-1}{2})^{2},$$
the corollary \ref{Poisson-Spheric} can be seen from the theorem \ref{Poisson-Hyperbolic} by an argument of analytic continuation.
\end{proof} 
Note that the result of the corollary agrees with the formula $(4.9)$ in (Taylor \cite{TAYLOR} p. 114).

\newcommand{\Addresses}{{
  \bigskip
  \footnotesize
  Adam Zakria,
 \textsc{Department of Mathematics,
 College of Arts and Sciences-Gurayat,
 Jouf University-Kingdom of Saudi Arabia.}\par\nopagebreak
 \textsc{ Department of Mathematics- Faculty of Sciences - University of Kordofan -Sudan\\}.
  \textit{E-mail address}: \texttt{adammath2020@gmail.com}\\
  \medskip
 Ibrahim Elkhalil,
 \textsc{Department of Mathematics,
 College of Arts and Sciences-Gurayat,
 Jouf University-Kingdom of Saudi Arabia.}\par\nopagebreak
 \textsc{Shendi University ,Faculte of science and Technologe, Departement of Mathematics ,Shendi Sudan}\\.
  \textit{E-mail address}: \texttt{i\_elkhalil33@yahoo.com}

  \medskip

  Mohamed Vall Ould Moustapha, \textsc{Department of Mathematic,
 College of Arts and Sciences-Gurayat,
 Jouf University-Kingdom of Saudi Arabia }\par\nopagebreak
\textsc{ Faculte des Sciences et Techniques
Universit\'e de  Nouakchott Al-Aasriya.
Nouakchott-Mauritanie}\\
  \textit{E-mail address}: \texttt{mohamedvall.ouldmoustapha230@gmail.com}}}

 \Addresses
\end{document}